\documentclass[11pt,a4paper]{article}
\usepackage{fancyhdr}
\usepackage{amsmath}
\usepackage{amsthm}
\usepackage{amssymb}
\usepackage{wasysym}
\usepackage{paralist}
\usepackage{makeidx}
\usepackage[all]{xy}
\usepackage{mathdots}
\usepackage{yhmath}

\usepackage{young}
\usepackage[vcentermath]{youngtab}

\usepackage{graphicx}
\usepackage{epstopdf}

\input xy

\xyoption{all}
\makeindex

\renewcommand{\det}{{\rm det}}
\newcommand{\ep}{{\epsilon}}
\newcommand{\de}{{\delta}}

\newcommand{\ind}{{\rm ind}}
\renewcommand{\vert}{\,|\,}
 \newcommand{\tbw}{\textstyle\bigwedge}
\newcommand{\zed}{{\mathbb Z}}
\newcommand{\Ext}{{\rm Ext}}
\newcommand{\End}{{\rm End}}
\newcommand{\Hom}{{\rm Hom}}
\newcommand{\cf}{{\rm cf}}
\newcommand{\soc}{{\rm soc}}
\renewcommand{\mod}{{\rm mod}}
\newcommand{\Map}{{\rm Map}}
\newcommand{\GL}{{\rm GL}}
\newcommand{\q}{\quad}
\newcommand{\St}{{\rm St}}
\newcommand{\hatZ}{{\hat Z}}
\newcommand{\hatL}{{\hat L}}
\newcommand{\hatQ}{{\hat Q}}
\newcommand{\barlambda}{{\bar \lambda}}
\newcommand{\barmu}{{\bar \mu}}
\newcommand{\Tr}{{\rm Tr}}
\newcommand{\padic}{{p-{\rm adic}}}
\newcommand{\bs}{\bigskip}
\newcommand{\id}{{\rm id}}
\newcommand{\opp}{{\rm opp}}
\renewcommand{\inf}{{\rm inf}}
\newcommand{\barB}{{\bar B}}
\newcommand{\barG}{{\bar G}}
\newcommand{\barnabla}{{\overline \nabla}}
\newcommand{\barL}{{\bar L}}
\newcommand{\sgn}{{\rm sgn}}

\newcommand{\Sym}{{\rm Sym}}
\newcommand{\ch}{{\rm ch\,}}

\renewcommand{\makeindex}{Index}
\newtheorem{definition}{Definition}[section]
\newtheorem{proposition}[definition]{Proposition}
\newtheorem{theorem}[definition]{Theorem}
\newtheorem{lemma}[definition]{Lemma}
\newtheorem{corollary}[definition]{Corollary}

\newtheorem{remark}[definition]{Remark}

\begin{document}

\parindent=0pt

\centerline{\bf Endomorphism Algebras of Some Modules for Schur Algebras}
\centerline{\bf  and Representation Dimension}

\bigskip

\centerline{Stephen Donkin and Haralampos Geranios}

\bigskip

{\it Department of Mathematics, University of York, York YO10 5DD}

{\tt stephen.donkin@york.ac.uk and }\\
{\tt   haralampos.geranios@york.ac.uk}

\bigskip

\centerline{11 February 2013}

\bigskip \bigskip

\section*{Abstract}

We consider the representation dimension, for fixed $n\geq2$, of ordinary and quantised Schur algebras $S(n,r)$ over a field $k$. For $k$ of positive characteristic $p$ we give a lower bound valid for all $p$. We also give an upper bound in the quantum case, when $k$ has characteristic $0$.  

\section{Introduction}

\quad In a recent paper \cite{MO}, Miemietz and Oppermann describe a lower bound for the representation dimension of Schur algebras $S(n,r)$ over a field of characteristic $p>0$. This bound is stated under the condition $p\geq 2n-1$.  We here give a bound valid for all $p$. Our main interest is in the classical case so the main body of the text is expressed in the classical context. However, we point out that our results are valid also in the quantum case in positive characteristic.  We also point out that for fixed $n$ there is an upper bound of the representation dimension of the $q$-Schur algebras $S_q(n,r)$ valid for all $r$ in the characteristic $0$  case.

\q 
There are certain similarities between our approach and that of Miemietz and Oppermann. In particular we follow their method of finding modules with endomorphism algebra isomorphic  to a truncated polynomial algebra (we work with injective modules and they with projective modules).  However, in detail our techniques are very different and our  main technique is to compare the representation theory of the general linear group with that of its infinitesimal subgroups.
Moreover, we avoid certain problems with the argument of Miemietz and Oppermann (see the Remark of Section 3 for details). However, as in \cite{MO}, we are still reliant on the  work of Bergh, \cite{Bergh}, Oppermann, \cite{Oppermann} and Rouquier, \cite{Rouquier1, Rouquier2}, relating representation dimension and certain endomorphism algebras via derived categories.

\q In Section 2 we deal with the preliminary background material and notation. In Section 3 we consider certain polynomial injective modules and their endomoprhism algebras. In Section 4 we consider further the endomorphism algebras of these modules and deduce that, for fixed $n\geq 2$,  the representation dimension of $S(n,r)$ grows with $r$. For the precise result, see Corollary 4.4.

\q  In the final section we work over a quantum general linear group at a non-zero parameter $q$.  In particular we  demonstrate a certain dichotomy.  We first consider the characteristic $0$ case and show, Proposition 5.2, that for $n$ (and $q$)  fixed the representation dimension of the $q$ Schur algebra $S_q(n,r)$ (as $r$ varies) is bounded above. We then consider the case in which $k$ has positive characteristic and show,  Theorem 5.3, that if $q\neq 1$ is a root of unity then, as in the classical case, the representation dimension of $S_q(n,r)$ grows with $r$.

\section{Preliminaries}

\bigskip

\quad We start with some standard combinatorics associated with the representation theory of general linear groups. We fix a positive integer $n$. We set  $X(n)=\zed^n$.   There is a natural partial order on $X(n)$. For $\lambda=(\lambda_1,\ldots,\lambda_n), \mu=(\mu_1,\ldots,\mu_n)\in X(n)$,  we write $\lambda\leq \mu$ if $\lambda_1+\cdots+\lambda_i\leq \mu_1+\cdots+\mu_i$ for $i=1,2,\ldots,n-1$ and $\lambda_1+\cdots+\lambda_n=\mu_1+\cdots+\mu_n$. We shall use the standard $\zed$-basis   $\ep_1,\ldots,\ep_n$ of   $X(n)$,  where $\ep_i=(0,\ldots,1,\ldots,0)$ (with $1$ in the $i$th position).   We shall also need the specific elements $\omega=(1,1,\ldots,1)$ and  $\de=(n-1,\ldots,1,0)\in X(n)$.  The symmetric group $W=\Sym(n)$ acts naturally on $X(n)$.  We write $w_0$ for the longest element of  $W$,  i.e., the element such that $w_0\lambda=(\lambda_n,\ldots,\lambda_1)$, for $(\lambda_1,\ldots,\lambda_n)\in X(n)$.  

\quad We write $X^+(n)$ for the set of dominant weights, i.e., the set of $\lambda=(\lambda_1,\ldots,\lambda_n)\in X(n)$ such that $\lambda_1\geq \cdots\geq \lambda_n$.  We write $\Lambda(n)$ for the set of polynomial weights, i.e., the set of $\lambda\in X(n)$ with all $\lambda_i\geq 0$.  We set  $\Lambda^+(n)=X^+(n)\bigcap \Lambda(n)$, the set of polynomial dominant weights. For $\lambda=(\lambda_1,\ldots,\lambda_n)\in \Lambda(n)$ we define  its degree   $|\lambda|=\lambda_1+\cdots+\lambda_n$ and if $\lambda\in \Lambda^+(n)$, define its  breadth $b(\lambda)=\lambda_1$.
For a positive integer $r$ we write  $\Lambda^+(n,r)$ for the set of all $\lambda\in \Lambda^+(n)$ such that $|\lambda|=r$. An element of $\Lambda^+(n,r)$ is called a partition of $r$ into at most $n$ parts.  For an integer $l\geq 2$  we say that $\lambda=(\lambda_1,\ldots,\lambda_n)\in\Lambda^+(n)$ is {\it column $l$-regular} if $\lambda_i-\lambda_{i+1}<l$ for $1\leq i\leq n-1$ and $\lambda_n<l$.

\quad We form the integral group ring $\zed X(n)$ of $X(n)$. This has $\zed$-basis of formal exponentials $e^\lambda$, $\lambda\in X(n)$, which  multiply according to the rule $e^\lambda e^\mu=e^{\lambda+\mu}$. The action of $W$ extends to an action on $\zed X(n)$ by ring automorphisms.  For $\lambda\in X(n)$ we write $s(\lambda)$ for the orbit sum $\sum_{\mu\in W\lambda} e^\mu$.

\quad  We fix  an algebraically closed field $k$ and a positive integer $n$.    We write $G$ or $G(n)$ for the general linear group  $\GL_n(k)$.  We write $k[G]$ for the coordinate algebra of $G$.  For $1\leq i,j\leq m$ let $c_{ij}$ denote the corresponding coordinate function, i.e., the function taking $g\in G$ to its $(i,j)$ entry. We set $A(n)=k[c_{11},\ldots,c_{nn}]$.  Then $A(n)$ is a free polynomial algebra in generators $c_{11},\ldots,c_{nn}$.  Moreover  $A(n)$ has a $k$-space decomposition 
$A(n)=\oplus_{r=0}^\infty A(n,r)$, where  $A(n,r)$ is the span of the all monomials  $c_{i_1j_2}c_{i_2j_2}\ldots c_{i_rj_r}$.

\q We write  $T$ or $T(n)$ for the subgroup of $G$ consisting of all invertible diagonal matrices.   For $\lambda=(\lambda_1,\ldots,\lambda_n)\in X(n)$ we also denote by $\lambda$ the multiplicative  character of $T(n)$ defined by $\lambda(t)=t_1^{\lambda_1}\ldots t_n^{\lambda_n}$, where $t_i$ denotes the $(i,i)$-entry of the diagonal matrix $t\in T(n)$, and in this way identify $X(n)$ with  the character group of $T(n)$. 
     For $\lambda\in X(n)$ we write $k_\lambda$ for the one dimensional $T$-module on which $t\in T$ acts as multiplication by $\lambda(t)$. 
The modules  $k_\lambda$, $ \lambda\in X(n)$,  form  a complete set of pairwise non-isomorphic rational $T(n)$-modules.  We write $B$, or $B(n)$ (resp. $B^+$ or $B^+(n)$) for the subgroup of $G$ consisting of all  lower (resp. upper)  triangular invertible matrices.  For $\lambda\in X(n)$, the action of $T(n)$ extends uniquely to a module action of $B(n)$ (resp. $B^+(n)$) on $k_\lambda$ and the modules $k_\lambda$, $\lambda\in X(n)$,  form  a complete set of pairwise non-isomorphic irreducible rational $B(n)$-modules (resp. $B^+(n)$-modules). 

\q For $\lambda\in X^+(n)$ we write $\nabla(\lambda)$ for the induced module $\ind_B^Gk_\lambda$.  The dual of a finite dimensional rational $G$-module $V$ will be denoted $V^*$. We also have the Weyl module $\Delta(\lambda)=\nabla(-w_0\lambda)^*$, for $\lambda\in X^+(n)$.   A good filtration (resp. Weyl filtration) of a rational $G$-module $V$ is a filtration $0=V_0\leq V_1\leq \cdots\leq V_r=V$ such that for each $0<i\leq r$  the section $V_i/V_{i-1}$ is either $0$ or isomorphic to $\nabla(\lambda^i)$ (resp. $\Delta(\lambda^i)$) for some $\lambda^i\in X^+(n)$. If $V$ has a good filtration (resp. Weyl filtration) then, for $\lambda\in X^+(n)$, the number of sections in the filtration isomorphic to $\nabla(\lambda)$ (resp. $\Delta(\lambda)$) is independent of the choice of good filtration (resp. Weyl filtration) and will be denoted $(V:\nabla(\lambda))$ (resp. $(V:\Delta(\lambda))$).

\quad For a rational $T$-module $V$ and $\lambda\in X(n)$ we have the $\lambda$ weight space $V^\lambda=\{v\in V\vert tv=\lambda(t)v \hbox{ for all } t\in T\}$.  A rational $T$-module $V$ decomposes as a direct sum of its weight spaces: $V=\bigoplus_{\lambda\in X(n)} V^\lambda$.  To a finite dimensional rational $T$-module $V$ we attach its character 
$$\ch V=\sum_{\mu \in X(n)} \dim V^\mu e^\mu\in\zed X(n)$$
For $\lambda\in X^+(n)$ the character $\chi(\lambda)=\ch \nabla(\lambda)=\ch \Delta(\lambda)$ is given by Weyl's character formula, [\cite{Jan}, II, 5.10 Proposition]. 
For each $\lambda\in X^+(n)$ there is an irreducible rational $G$-module $L(\lambda)$, with unique highest weight $\lambda$ occurring with multiplicity one,  and
the modules  $L(\lambda)$, $\lambda\in X^+(n)$,  form  a complete set of pairwise non-isomorphic irreducible rational $G$-modules.   For $r\geq 1$ we write $\St_r$ for the $r$th Steinberg module $L((p^r-1)\delta)$. We also write simply $\St$ for $\St_1$. We shall write $D$ for the determinant module, i.e, the one dimensional $G$-module on which $g\in G$ acts as multiplication by $\det(g)$.  For a finite dimensional $G$-module and $r\geq 0$ we write $V^{\otimes r}$ for the $r$-fold tensor product $V\otimes\cdots\otimes V$. We write $V^{\otimes -r}$ for the dual module $V^*\otimes \cdots\otimes V^*$. In particular  we have $D^{\otimes r}=L(r,r,\ldots,r)$, for any integer $r$.  We write $E$ for the space of column vectors of length $n$. Then $E$ is a $G$-module (the natural module) with action by matrix multiplication. For $1\leq r\leq n$ the exterior power $\tbw^r E$ is irreducible of highest weight $(1,1,\dots,1,0,\ldots,0)$ (where $1$ occurs $r$ times). So we have $\tbw^rE=L(1,1,\ldots,1,0,\ldots,0)$ and in particular $D=\tbw^n E$.
 For a finite dimensional rational $G$-module and $\lambda\in X^+(n)$ we write $[V:L(\lambda)]$ for the multiplicity of $L(\lambda)$ as a composition factor of $V$. 
 
 \q A finite dimensional rational module $V$ which has a filtration by induced modules and a filtration by Weyl modules is called a tilting module. For each $\lambda\in X^+(n)$ there exists an indecomposable  tilting module $M(\lambda)$ which has a weight $\lambda$ occurring with multiplicity one and all other weights smaller than $\lambda$. Moreover, the modules $M(\lambda)$, 
$\lambda\in X^+(n)$,  form a  complete set of pairwise non-isomorphic  indecomposable tilting modules.
 
 \q We write $\Map(G,k)$ for the space of all $k$-valued functions on $G$.  Let $V$ be a finite dimensional rational module with basis $v_1,\ldots,v_m$. The coefficient functions $f_{ij}$, $1\leq i,j\leq m$, are defined by the equations
 $$gv_i=\sum_{j=1}^m f_{ji}(g)v_j$$
 for $g\in G$, $1\leq i\leq m$.  The coefficient space $\cf(V)$ is the $k$-span of the coefficient functions $f_{ij}$, $1\leq i,j\leq m$.

\q A $G$-module $V$ is called polynomial if $\cf(V)\leq A(n)$ and polynomial of degree $r$ if $\cf(V)\leq A(n,r)$.  A polynomial $G$-module $V$ has a unique module decomposition 
$$V=\bigoplus_{r=0}^\infty V(r)$$
where $V(r)$ is polynomial of degree $r$. The coordinate algebra $k[G]$ has a natural Hopf algebra structure and each space $A(n,r)$ is a subcoalgebra. The dual space $S(n,r)$ has a natural algebra structure. The algebras $S(n,r)$ are called Schur algebras. The category of polynomial modules of degree $r$ is naturally equivalent to the category of $S(n,r)$-modules.

\q The modules $L(\lambda)$, $ \lambda\in \Lambda^+(n)$, form a complete set of pairwise non-isomorphic irreducible polynomial modules.  
For $\lambda\in \Lambda^+(n)$ we write $I(\lambda)$ for the injective hull of $L(\lambda)$ in the category of polynomial modules. Then $I(\lambda)$ is finite dimensional, indeed it is may be identified with the injective hull of $L(\lambda)$ as a module for the Schur algebra $S(n,r)$, where $r=|\lambda|$. For further details of the category of polynomial modules in the spirit of this paper see for example, \cite{Green}, \cite{Do2},\cite{DeVisscherDonkin}.  We shall use, without further reference the fact that a rational $G$-module is polynomial if and only if all of its composition factors are of the form $L(\lambda)$, $\lambda\in \Lambda^+(n)$ - this follows from  \cite{Do1},  Section 1, (1) for example. We shall call a $G$-module polynomially injective if it is polynomial and injective in the category of polynomial modules.

\quad Now suppose that $k$ has prime characteristic $p$. Let $r\geq 1$. We write $X_m(n)$ for the set of all column $p^m$-regular dominant weights. We have the Frobenius morphism $F:G\to G$, taking an invertible matrix $(a_{ij})$ to $(a_{ij}^p)$.  If $V$ is a rational $G$-module  affording  representation $\pi:G\to \GL(V)$ then write $V^F$ for the  vector space $V$ viewed as a $G$-module via the representation $\pi\circ F$.  We write $G_m$ (resp. $B_m$, resp. $B^+_m$) for the $m$th infinitesimal subgroup of $G$ (resp. $B$, resp. $B^+$).  The modules $L(\lambda)$, $\lambda\in X_m(n)$, form a complete set of pairwise non-isomorphic irreducible $G_m$-modules. Given a rational $G$-module $V$ the fixed subspace $V^{G_m}$ is a $G$-submodule. A rational $G$-module $V$ which is trivial as a $G_m$ is isomorphic to $Z^{F^m}$ for a uniquely determined (up to isomorphism) rational $G$-module $Z$.   For a rational $G$-module $V$ we have the exact sequence,  [\cite{Jan}, II 9.23],

\begin{align*}
0&\to H^1(G/G_m,V^{G_m})\to H^1(G,V)\to H^1(G_m,V)^{G/G_m}\cr
&\to H^2(G/G_m,V^{G_m})\to H^2(G,V).
\end{align*}
This will be particularly useful for us in conjunction with the isomorphism 
$$H^i(G/G_m,Z^{F^m})\to H^i(G,Z)$$
for $Z\in \mod(G)$ and $i\geq 0$ (which follows from the fact that $F^m:G\to G$ induces an isomorphism $G/G_m\to G$). We shall refer to the first property  as the $5$-term exact sequence and use the second property without further reference. 

\q Finally, for $\lambda\in X_m(n)$, $\mu\in X^+(n)$, we have $L(\lambda+p^m\mu)\cong L(\lambda)\otimes L(\mu)^{F^m}$ (see  [\cite{Jan}, II Proposition 3.16]). We shall also use this without further reference.

\section{Some injective indecomposable modules}

\bigskip

\bf Remark\q\rm We place the emphasis throughout on injective modules rather than the projective modules as used in \cite{MO}. We note that there are certain problems with the argument of Miemietz and Oppermann, \cite{MO}.  In particular  \cite{MO}, Lemma 4.5 is not correct.  (For a counterexample take $n=2$, $m=1$, $\lambda^0=0$  and $\lambda^1=(1,0)$.   The  assertion is then that $E^F$ is a certain tilting module but it has highest weight $(p,0)$ and the tilting module of this highest weight must have the corresponding induced module, i.e the symmetric power $S^pE$,  as a section. This is impossible since $E^F$ has dimension $2$ and $S^pE$ has dimension $p+1$.) Furthermore, \cite{MO}, Lemma 4.5 is used in the arguments to justify \cite{MO}, Lemmas 4.6 and 4.8. (However, we know that the statements are true because of Lemma 3.4 below which we prove using our infinitesimal methods.) Lemmas 4.6 and 4.8 of \cite{MO} are used to reach the main goal of the paper \cite{MO}, namely Theorem 4.13.

\q  For the most part our arguments proceed via the relationship between the representation theory of $G$ and $G_m$ and are very different from those in \cite{MO}. In  order to give   a clear line of development we  give complete proofs but note that the proof of  Lemma 3.4 below  has features in common with the proof of \cite{MO}, Lemma 4.6.

\bigskip

\q For a $G$-module, resp. $G_m$-module,  $M$ we write $\soc_G(M)$ (resp. $\soc_{G_m}(M)$) for the socle of $M$  as a $G$-module  (resp. $G_m$-module). 

\bigskip

\begin{lemma}

Let $M,N$ be rational $G$-modules. If $M$ has $G_m$-socle $L(\lambda)$, with $\lambda\in X_m(n)$, then $M\otimes N^{F^m}$ has $G$-socle $L(\lambda)\otimes \soc_G(N)^{F^m}$.

\end{lemma}

\begin{proof}

Let $Z=M\otimes N^{F^m}$.  Using   [\cite{Jan}, II 3.16]  we have 

$$\soc_GZ \cong \bigoplus_{\xi \in X_m(n)}L(\xi)\otimes \soc_G(\Hom_{G_m}(L(\xi),Z)).$$

However, as a $G_m$-module $Z$ is a direct sum of copies of $M$ and so its  $G_m$-socle is a direct sum of copies of $L(\lambda)$.  Hence a non-zero term in the above can only occur in case $\xi=\lambda$. Moreover, we have 
$$\Hom_{G_m}(L(\lambda),Z)=\Hom_{G_m}(L(\lambda),M)\otimes N^{F^m}\cong N^{F^m}.$$
The result follows.

\end{proof}

\begin{lemma}  Let $\lambda\in X_m(n)$ and $\mu\in \Lambda^+(n)$.

(i) The $G$-module  $I(\lambda+p^m\mu)$ has  $G_m$-socle $L(\lambda)\otimes I(\mu)^{F^m}$.

(ii) If $I$ is a finite dimensional polynomial $G$-module such that $\soc_{G_m}(I)=L(\lambda)$ (as $G$-modules) and $I|_{G_m}$ is injective. Then $I=I(\lambda)$. 

(iii) For $\mu\in \Lambda^+(n)$ the $G$-module $I(\lambda)\otimes I(\mu)^{F^m}$ has simple socle $L(\lambda+p^m\mu)$ and if  $I(\lambda)|_{G_m}$ is injective then 
$$I(\lambda+p^m\mu)=I(\lambda)\otimes I(\mu)^{F^m}.$$

\end{lemma}

\begin{proof} (i)   Using   [{\cite{Jan}, II 3.16, (1)}], we have an isomorphism of $G$-modules,

\[\soc_{G_m}I(\lambda+p^m\mu)\cong \bigoplus_{\xi \in X_m(n)}L(\xi)\otimes \Hom_{G_m}(L(\xi),I(\lambda+p^m\mu)).\]

Since $I(\lambda+p^m\mu)$ has simple $G$-socle $L(\lambda)\otimes L(\mu)^{F^m}$ we must have 

$$\soc_{G_m}I(\lambda+p^m\mu)=L(\lambda)\otimes V^{F^m}$$
for some $G$-module $V$ with simple socle $L(\mu)$. Thus $V$ embeds in $I(\mu)$ and so $\soc_{G_m}(I(\lambda+p^m\mu))$ embeds in $L(\lambda)\otimes I(\mu)^{F^m}$. 
 On the other hand $L(\lambda)\otimes I(\mu)^{F^m}$ has simple socle $L(\lambda)\otimes L(\mu)^{F^m}$ by the above Lemma. So $L(\lambda)\otimes I(\mu)^{F^m}$ embeds in $I(\lambda+p^m\mu)$ and hence in $\soc_{G_m}(I(\lambda+p^m\mu))$. 
  So we must have $V=I(\mu)$ and we are done.

(ii) Note that $I$ is indecomposable so is polynomial of some degree $r$, say.    To prove that $I$ is polynomially injective we show  that $\Ext^1_G(L(\tau),I)=0$ for all $\tau\in \Lambda^+(n,r)$. We write $\tau=\xi+p^m\nu$ with $\xi\in X_m(n)$, $\nu\in \Lambda^+(n)$. Then $L(\tau)=L(\xi)\otimes L(\nu)^{F^m}$.  Since $I|_{G_m}$ is injective the $5$-term exact sequence gives 
$$\Ext^1_G(L(\xi)\otimes L(\nu)^{F^m},I)=H^1(G/G_m,\Hom_{G_m}(L(\xi)\otimes L(\nu)^{F^m},I)).$$
The $G_m$-socle of $I$ is $L(\lambda)$ so if $\Hom_{G_m}(L(\xi)\otimes L(\nu)^{F^m},I)\neq 0$ then  $\xi=\lambda$ and then $|\xi+p^m\nu|=r=|\lambda|$ implies that $|p^m\nu|=0$ so that $\nu=0$ and the above becomes $H^1(G/G_m,\Hom_{G_m}(L(\lambda),I))=H^1(G,k)=0$. 

(iii) The first assertion follows immediately from  Lemma 3.1

\q We now assume $I(\lambda)|_{G_m}$ is injective. To conclude we need to show that $I(\lambda)\otimes I(\mu)^{F^m}$ is polynomially injective and we do this by proving that 

$\Ext^1_G(L(\tau),I)=0$ for all $\tau\in \Lambda^+(n,r)$.  As usual we write $\tau=\xi+p^m\nu$, with $\xi\in X_m(n)$, $\nu\in \Lambda^+(n)$.  We use the $5$ term exact sequence, as in the proof of (ii), and we deduce that 
\begin{align*}
\Ext^1_G&(L(\tau),I(\lambda)\otimes I(\mu)^{F^m})\cr
=H^1&(G/G_m,(\Hom_{G_m}(L(\xi)\otimes L(\nu)^{F^m},I(\lambda)\otimes I(\mu)^{F^m})).
\end{align*}
But if $\Hom_{G_m}(L(\xi)\otimes L(\nu)^{F^m},I(\lambda)\otimes I(\mu)^{F^m})\neq 0$ then 
 we have $\xi=\lambda$ and 
\begin{align*}
H^1(&G/G_m,(\Hom_{G_m}(L(\xi)\otimes L(\nu)^{F^m},I(\lambda)\otimes I(\mu)^{F^m}))\cr
&=H^1(G/G_m,(L(\nu)^{F^m})^*\otimes I(\mu)^{F^m})=H^1(G,L(\nu)^*\otimes I(\mu))\cr
&=\Ext^1_G(L(\nu),I(\mu))=0
\end{align*}
since $|\nu|=|\mu|$ and $I(\mu)$ is polynomially injective.

\end{proof}

\begin{remark} It would be interesting to know precisely for which $\lambda\in \Lambda^+(n)$ the module $I(\lambda)$ is injective as a $G_m$-module. We will encounter this property in various special cases  in what follows.

\end{remark}

\bigskip

\begin{lemma}
Let $h\geq 0$. Let $\lambda^i  \in X_m(n)$,  for $0\leq i<h$,  and $\gamma \in \Lambda^+(n)$.  We put $\lambda=\sum_{0\leq i<h} p^{mi}\lambda^i +p^{mh}\gamma$ and $I_h(\lambda)=\bigotimes_{0\leq i<h} I(\lambda^i)^{F^{mi}}\bigotimes I(\gamma)^{F^{mh}}$.

(i) The  module $I_h(\lambda)$ has socle $L(\lambda)$.

(ii) We have 
$$[I_h(\lambda):L(\lambda)]=[I(\lambda^0):L(\lambda^0)]\ldots [I(\lambda^{h-1}):L(\lambda^{h-1})][I(\gamma):L(\gamma)].$$

(iii) We have 

 \[\End_G(I_h(\lambda))=\bigotimes_{0\leq i<h} \End_G(I(\lambda^i)^{F^{mi}})\bigotimes\End_G( I(\gamma)^{F^{mh}}). \]
\end{lemma}

\begin{proof}

(i) This  follows   from  Lemma 3.1 with $M=I(\lambda^0)$, $N=I(\lambda^1)\otimes\cdots\otimes I(\gamma)^{F^{m(h-1)}}$ and induction.

(ii)  The result is trivial for $h=0$. Now assume $h\geq 1$. Let $J=I(\lambda^1)\otimes\cdots\otimes I(\gamma)^{F^{m(h-1)}}$. Let $0=Z_0<Z_1<\cdots<Z_t$ be a composition series of $I(\lambda^0)$, with $Z_i/Z_{i-1}\cong L(\xi^i)\otimes L(\nu^i)^{F^m}$, with $\xi^i\in X_m(n)$ and $\nu^i\in \Lambda^+(n)$, for $1\leq i\leq t$.  Then we have a filtration $Z_0\otimes J^{F^m}< Z_1\otimes J^{F^m} <\cdots < Z_t\otimes J^{F^m}$ of $I_h(\lambda)$.  Hence we have
$$[I_h(\lambda):L(\lambda)]=\sum_{i=1}^t  [(Z_i/Z_{i-1})\otimes J^{F^m}:L(\lambda)].$$

\q But now $(Z_i/Z_{i-1})\otimes J^{F^m}$, as a $G_m$-module, is a direct sum of copies of $L(\xi^i)$ and $L(\lambda)$, as a $G_m$-module, is a direct sum of copies of $L(\lambda^0)$.  Hence $[(Z_i/Z_{i-1})\otimes J^{F^m}:L(\lambda)]\neq 0$ implies that $\xi^i=\lambda^0$, and since $|\xi^i+p^m\nu^i|=|\lambda^0|$ in this case we must have $\nu^i=0$ and 
$$[(Z_i/Z_{i-1})\otimes J^{F^m}:L(\lambda)]=[L(\lambda^0)\otimes J^{F^m}:L(\lambda^0)\otimes L(\mu)^{F^m}]=[J:L(\mu)]$$
where $\mu=\lambda^1+\cdots+p^{m(h-2)}\lambda^{h-1}+p^{m(h-1)}\gamma$.   
Hence we have
$$[I_h(\lambda):L(\lambda)]=[I(\lambda^0):L(\lambda^0)][J:L(\mu)]$$
i.e., 
$$[I_h(\lambda):L(\lambda)]=[I(\lambda^0):L(\lambda^0)][I_{h-1}(\mu):L(\mu)]$$
and the result follows by induction.

(iii)   We identify

$$\bigotimes_{0\leq i<h} \End_G(I(\lambda^i)^{F^{mi}})\bigotimes\End_G( I(\gamma)^{F^{mh}})$$

with a subalgebra of  $\End_G(I_h(\lambda))$ via the natural embedding.
 It suffices to show that  the dimension of the second algebra is bounded above by the dimension of the first.
By (i), $I_h(\lambda)$ has simple $G$-socle $L(\lambda)$ and so, by left exactness of $\Hom_G(-,I_h(\lambda))$ we have that $\dim \Hom_G(V,I_h(\lambda))\leq [V:L(\lambda)]$, for $V\in \mod(G)$. Hence the dimension $\End_G(I_h(\lambda))$ is bounded above by the composition multiplicity $[I_h(\lambda):L(\lambda)]$.  The result now follows by part (ii).

\end{proof}

\bigskip

\q We introduce some additional notation for use in the proof of the next result. For $\lambda\in X(n)$ there exists a unique (up to isomorphism) irreducible $G_1T$-module $\hatL_1(\lambda)$ with highest  weight $\lambda$. We write $\hatQ_1(\lambda)$ for the injective hull of $\hatL_1(\lambda)$ as a $G_1T$-module.  
For $\lambda\in X(n)$ we have the induced modules $\hatZ_1'(\lambda)=\ind_{B_1T}^{G_1T}k_\lambda$ and $\hatZ_1(\lambda)=\ind_{B_1^+T}^{G_1T}k_\lambda$.  
The character of these modules is given by 
$$\ch \hatZ_1'(\lambda)=\ch \hatZ_1(\lambda)=e^{\lambda-(p-1)\delta}\chi((p-1)\delta).$$

The module $\hatQ_1(\lambda)$ admits a filtration with sections $\hatZ_1'(\mu)$, $\mu\in X(n)$,  and also a filtration with sections $\hatZ_1(\mu)$, $\mu\in X(n)$. 

\q For $\lambda\in X_1(n)$ we have $L(\lambda)|_{G_1T}=\hatL_1(\lambda)$.  We write $Q_1(\lambda)$ for the restriction of $\hatQ_1(\lambda)$ to $G_1$. Then $Q_1(\lambda)$ is the injective hull of $L(\lambda)$, as a $G_1$-module, for $\lambda\in X_1(n)$. For $\mu\in X(n)$ we write   $Z_1'(\mu)$ (resp. $Z_1(\mu)$) for the restriction of $\hatZ_1'(\mu)$ (resp. $\hatZ_1(\mu)$) to $G_1$. The module $Q_1(\lambda)$ admits a filtration with sections $Z_1'(\mu)$, $\mu\in X(n)$,  and also a filtration with sections $Z_1(\mu)$, $\mu\in X(n)$. 

\q For a full discussion of these properties see,  [\cite{Jan}, II Chapter 9].

\begin{remark}
Let  $m\geq 1$ and  $\lambda\in X_m(n)$ with  $b(\lambda)<p^m$. It is easy to check that   $(p^m-1)\delta+w_0\lambda \in X_m(n)$. 
\end{remark}

\bigskip

  Lemma 3.6 (iv)  and Lemma 4.1(ii) follow from (a very special case of)  a result of Andersen, Jantzen and Soergel, [\cite{AJS}, 19.4 Theorem a)]. We provide an independent proof in our  case using methods in keeping with the spirit of this paper.

\bigskip

\begin{lemma}
Suppose $\lambda\in X_1(n)$ and  $b(\lambda)<p$.  Then:

(i) $M((p-1)\delta+\lambda)=I((p-1)\delta+w_0\lambda)$;

(ii) $I((p-1)\delta+w_0\lambda)|_{G_1T}=\hatQ_1((p-1)\delta+w_0\lambda)$;

(iii) $\End_{G_1}(I((p-1)\delta+w_0\lambda))=\End_G(I((p-1)\delta+w_0\lambda))$; and 

(iv)  $\End_G(M((p-1)\de+\lambda)$ is a symmetric algebra of dimension $|W\lambda|$.

\end{lemma}

 \begin{proof}

\q  Let $M$ denote an indecomposable summand of $\St\otimes L(\lambda)$  which has the highest weight $(p-1)\delta+\lambda$. 
We shall use the arguments of  [\cite{Do3}, p. 236-7], proved for $G$ a semisimple group, but the conversion to the context here $G={\rm GL}_n(k)$ is routine.  (See also [\cite{Jan},II, 11.10].) It is shown that $M=M((p-1)\delta+\lambda)$, that $\ch M=\chi((p-1)\delta)s(\lambda)$, where $s(\lambda)$ is the orbit sum $\sum_{\xi\in W\lambda}e^\xi$,  and $M|_{G_1T}=\hatQ_1((p-1)\delta+w_0\lambda)$. 
It now follows from Lemma 3.2(ii) that $M=I((p-1)\delta+w_0\lambda)$. We have now proved (i) and (ii).

\q By character considerations,  $M|_{G_1T}=\hatQ_1((p-1)\delta+w_0\lambda)$ has a filtration with sections $\hatZ_1'((p-1)\delta+\mu)$, $\mu\in W\lambda$  (each occurring once) and a filtration with sections $\hatZ_1((p-1)\delta+\tau)$, $\tau\in W\lambda$  (each occurring once). Hence $M|_{G_1}$ has a filtration with sections $Z_1'((p-1)\delta+\mu)$, $\mu\in W\lambda$  (each occurring once) and a filtration with sections $Z_1((p-1)\delta+\tau)$, $\tau\in W\lambda$  (each occurring once). 
Moreover, by [\cite{Jan}, II, 9.9 Remark 1], for $\mu,\tau\in X(n)$, we have

$$\Ext^i_{G_1}(Z_1'(\mu),Z_1(\tau))=\begin{cases} k, & {\rm if}\  \mu-\tau\in pX(n) \ {\rm and }\  i=0;\cr
0, &\rm{otherwise}
\end{cases}
$$

It follows that, for $\mu,\tau\in W\lambda$ we have
$$\Ext^i_{G_1}(Z_1'((p-1)\delta+\mu),Z_1((p-1)\delta+\tau))=\begin{cases} k, & {\rm if}\  \mu=\tau \ {\rm and }\  i=0;\cr
0, &\rm{otherwise.}
\end{cases}
$$
Hence we have
\begin{align*}
\dim \End_{G_1}(M)&=\sum_{\mu,\tau\in W\lambda} \dim \Hom_{G_1}((Z_1'((p-1)\delta+\mu),Z_1((p-1)\delta+\tau))
\cr
&=|W\lambda|.
\end{align*}

\q However, we may make a similar analysis of $\End_G(M)$. By Brauer's formula, [\cite{Jan}, II, 5.8 Lemma b)],  we have 
$$\ch M=\chi((p-1)\delta)s(\lambda)=\sum_{\mu\in W\lambda} \chi((p-1)\delta+\mu).$$
Since $M$ is a tilting module there is a  $G$-module filtration with sections   $\Delta((p-1)\delta+\mu)$, $\mu\in W\lambda$  (each occurring once) and a filtration with sections $\nabla((p-1)\delta+\tau)$, $\tau\in W\lambda$  (each occurring once).  Moreover, by [\cite{Jan}, II, 4.13 Proposition], for $\mu,\tau\in X^+(n)$, we have
$$\Ext^i_G(\Delta(\mu),\nabla(\tau))=\begin{cases} k, & {\rm if}\  \mu=\tau  \ {\rm and }\  i=0;\cr
0, &\rm{otherwise.}
\end{cases}
$$
Hence we have
$$\dim \End_G(M)=\sum_{\mu,\tau\in W\lambda} \dim \Hom_G((\Delta((p-1)\delta+\mu),\nabla((p-1)\delta+\tau))=|W\lambda|.$$

\q Since $\End_G(M)\leq \End_{G_1}(M)$ we must have $\End_G(M)= \End_{G_1}(M)$. This proves  (iii) and part of (iv). Finally, putting $H={\rm SL}_n(k)$ and writing $H_1$ for the first infinitesimal subgroup we have $\End_{G_1}(M)=\End_{H_1}(M)$ as the scalar matrices act on $M$ via scalar multiplication. Moreover the restriction of an injective $G_1$-module  to $H_1$ is an injective module.  The category of $H_1$-modules is equivalent to the category of modules for the algebra of distribution algebra of $H_1$, i.e., the restricted enveloping algebra $U_1$ of the Lie algebra ${\rm sl}_n(k)$.   However, since $U_1$ is a symmetric algebra, \cite{Hum}, a finite dimensional injective module for a symmetric algebra is  also projective and has symmetric endomorphism algebra so we are done.

\end{proof}

  \begin{remark}

The above Lemma is consistent with the conjecture, [\cite{DoTiltZeit}, (2.2)] that a projective indecomposable $G_1$-module (for $G$ a semisimple, simply connected group) is the restriction to $G_1$ of a certain tilting module. It is also consistent with the conjecture [\cite{DeVisscherDonkin},] Conjecture 5.1] which attempts to describe the $G$-modules both injective and projective in the polynomial category.

\end{remark}

\bigskip

\begin{proposition}
Let $\lambda^i\in  X_1(n)$ with $b(\lambda^i)<p$, for $0\leq i\leq h-1$, and let $\gamma\in X^+(n)$. Then  the $G$-module 
\[I((p-1)\delta+w_0\lambda^0)\otimes\dots\otimes I((p-1)\delta+w_0\lambda^{h-1})^{F^{(h-1)}}\otimes I(\gamma)^{F^{h}}\]
 is isomorphic to   $I((p^h-1)\delta+w_0\sum_{0\leq i\leq h-1}p^i\lambda^i+p^h\gamma)$.
\end{proposition}

\begin{proof} This follows by induction, using  Lemma 3.6(ii) and  Lemma 3.2(iii). 
\end{proof}

\bf Notation\q\rm  Let $\lambda\in \Lambda^+(n)$ and write $\lambda=\sum_{j=0}^m p^j\lambda^j$ with all $\lambda^j\in X_1(n)$.  We define the  $p$-adic breadth  $b_\padic(\lambda)$ to be  maximum value of $b(\lambda^j)$, $j=0,\ldots,m$.

\begin{corollary}
Let $\lambda^0,\lambda^1,\dots,\lambda^{h-1}\in X_m(n)$ and $\gamma\in \Lambda^+(n)$ with all $b_\padic(\lambda^i)<p$,   then

$$\bigotimes_{i=0}^{h-1} I((p^m-1)\delta+w_0\lambda^i)^{F^{mi}}\otimes I(\gamma)^{F^{mh}}$$

 is the injective hull of 
$L((p^{mh}-1)\delta+w_0\sum_{0\leq i\leq h-1}p^{mi}\lambda^i+p^{mh}\gamma)$ in the category of polynomial $G$-modules.

\q Moreover,  we have
\begin{align*}&\End_G(I(p^{mh}-1)\delta+w_0\sum_{0\leq i\leq h-1}p^{mi}\lambda^i+p^{mh}\gamma)\cr
&=\bigotimes_{i=0}^{h-1} \End_G(I(p^m-1)\delta+w_0\lambda^i))\otimes \End_G(I(\gamma))
\end{align*}

\end{corollary}

\begin{proof}
This follows immediately from  Lemma 3.4(iii)  and Proposition 3.8. 
\end{proof}

\section{Some Endomorphism Algebras}

\bigskip

For $\lambda=(\lambda_1,\ldots,\lambda_n)\in\Lambda^+(n)$ we set $\barlambda=(\lambda_1,\ldots,\lambda_{n-1})\in \Lambda^+(n-1)$.  In the proof following we write $\delta_n$ for $(n-1,n-2,\ldots,1,0)$, write $M_n(\lambda)$ for the tilting module of highest  weight $\lambda\in \Lambda^+(n)$, etc. to emphasize dependence on $n$.

\begin{lemma} (i) For $n\geq 2$ and $\lambda\in \Lambda(n)$ there is a surjective algebra homomorphism
$\End_{G(n)}(M_n(\lambda))\to \End_{G(n-1)}(M_{n-1}(\barlambda))$.

(ii) For  $1\leq a<p$ we have  $\End_{G(n)}(I((p-1)\delta_n+a\epsilon_n))=k[x]/(x^n)$.

\end{lemma}

\begin{proof} (i) We denote by $\Sigma$ the subset $\{\ep_1-\ep_2,\ep_2-\ep_3,\ldots,\ep_{n-2}-\ep_{n-1}\}$ of the set of simple roots 
$\Pi=\{\ep_1-\ep_2,\ldots,\ep_{n-1}-\ep_n\}$.  Let $G_\Sigma$ be the corresponding Levi subgroup.  Thus $G_\Sigma=H\times Z$, where $H$ is the subgroup of $G$ consisting of all invertible matrices  with $(i,j)$ entry $0$ for $i=n$ or $j=n$ and $(i,j)\neq (n,n)$ and $(n,n)$-entry $1$,  and $Z$ is the subgroup consisting of all invertible scalar matrices with $(i,i)$ entry $1$ for $i<n$. 
Then $G_\Sigma$ has maximal torus $T(n)$ and system of positive roots $\Sigma$. The set of dominant weights $X^+(\Sigma)$ consists of the elements $\lambda=(\lambda_1,\ldots,\lambda_n)$ such that $\lambda_1\geq \cdots\geq \lambda_{n-1}$.  For $\lambda\in X^+(\Sigma)$ we write $L_\Sigma(\lambda)$ for the simple $G_\Sigma$-module of highest  weight $\lambda$, write $\nabla_\Sigma(\lambda)$ for the corresponding induced module and $M_\Sigma(\lambda)$ for the corresponding tilting module.  

\q We have the truncation functor $\Tr_\Sigma^\lambda:\mod(G)\to\mod(G_\Sigma)$ as in \cite{DoTiltZeit}.  This functor induces an epimorphism $\End_G(M_n(\lambda))\to \End_{G_\Sigma}(M_\Sigma(\lambda))$, [\cite{DoTiltZeit}, (1.5) Proposition].  However, identifying $G(n-1)$ with $H$ in the obvious way we have $G_\Sigma=G(n-1)\times Z$ and so $M_{G_\Sigma}(\lambda)=M_{n-1}(\barlambda)\otimes L$, where $L$ is a one dimensional $Z$-module (of weight given by $\lambda_n$).  Hence we have $\End_{G_\Sigma}(M_\Sigma(\lambda))=\End_{G(n-1)}(M_{n-1}(\barlambda))$ and we are done.

(ii) We argue by induction on $n$. For $n=1$ this is clear. Now assume $n>1$ and the result holds for $n-1$. We have $I((p-1)\delta_n+a\ep_n)=M_n((p-1)\delta_n+a\ep_1)$ by Lemma  3.6(i).
We put $\mu=(p-1)\delta_n+a\ep_1$ and 
$$A=\End_{G(n)}(M_n((p-1)\delta_n+a\ep_1)).$$
  By (i) we have an epimorphism $A\to \End_{G(n-1)}(M_{n-1}(\barmu))$.  Now \\
  $\bar\mu=(p-1)\delta_{n-1}+a\ep_1+(p-1)\omega_{n-1}$.  But  
$$M_{n-1}(\barmu)=M_{n-1}((p-1)\delta_{n-1}+a\ep_1)\otimes D_{n-1}^{\otimes (p-1)}$$
 and since $D_{n-1}^{\otimes (p-1)}$ is one dimensional, we have
 $$\End_{G(n-1)}(M_{n-1}(\bar\mu))=\End_{G(n-1)}(M_{n-1}((p-1)\delta_{n-1}+a\ep_1)).$$
    By the inductive assumption we therefore have an epimorphism\\
        $A\to k[x]/(x^{n-1})$.  Moreover, by 
Lemma 3.6(iv), the dimension of $A$ is $n$ so the kernel of this epimorphism, say $U$,  is one dimensional.  We write $U=ku$.  Let $y\in A$ be an element mapping to $x+(x^{n-1})$. Then the elements $1,y,\ldots, y^{n-2}$ are linearly independent and $y^{n-1}=cu$ for some scalar $c$. If $c\neq 0$ then $A$  is isomorphic to the truncated polynomial algebra $k[x]/(x^n)$. So we may assume $y^{n-1}=0$. Since $u$ is in the socle of $A$ we have $yu=uy=0$. Now   $u$ and $y^{n-2}$ are independent elements of the socle of $A$, but this is impossible since $A$ is a local algebra and symmetric, by Lemma 3.6(iv).

\end{proof}

\q We shall need to discuss at some length modules whose endomorphism algebra is a truncated polynomial algebra. To facilitate the  discussion in this paper  we  now introduce appropriate terminology.

\bs

\bf Definition\q\rm We say a finite dimensional  polynomial $G$-module $M$ is {\it admissible} of index $r$ if its endomorphism algebra is isomorphic to the  truncated polynomial algebra $k[x_1,\ldots,x_r]/(x_1^n,\ldots,x_r^n)$.  We note that if polynomial $G$-modules $M_1,M_2$ are admissible with indices $r,s$ and $\End_G(M_1\otimes M_2)=\End_G(M_1)\otimes \End_G(M_2)$ then $M_1\otimes M_2$ is admissible of index $r+s$.

\bs

\q We assume from now on that $n\geq 2$. 

\q  In the proof of the next result we shall use the following general result, [\cite{Jan}, II proof of proposition 11.6] : if 
$H$ is a group-scheme over a field $k$  and $N$ is a normal subgroup scheme,  if $U$ is an $H$-module which is  injective as an $N$-module and if $V$ an $H$-module which is injective as an $H/N$-module then $U\otimes V$ is an injective $H$-module. 

\begin{proposition} Let  $m\geq 1$.

(i) For $0\leq a \leq p^m-1$ the module $I((p^m-1)\delta+a\ep_n)$ is injective as a $G_m$-module.

(ii) For $0\leq a < p^m-1$ we have

\[I((p^m-1)\delta+a \epsilon_n+\omega)\cong I((p^m-1)\delta+a \epsilon_n)\otimes D.\]

\end{proposition}

\begin{proof}

(i) We write $a=a_0+pa_1$ with $0\leq a_0\leq p-1$, $0\leq a_1\leq p^{m-1}-1$ and put $\lambda=(p-1)\delta+a_0\ep_n$ and $\mu=(p^{m-1}-1)\delta+a_1\ep_n$. Then $I((p-1)\delta+a_0\ep_n)$ is injective as a $G_1$-module by Lemma 3.6(ii). Thus we have $I((p^m-1)\delta+a\ep_n)=I(\lambda)\otimes I(\mu)^F$.  Now by induction we may assume that $I(\mu)$ is injective as a $G_{m-1}$-module. The Frobenius morphism  $F:G\to G$ induces an isomorphism $G_m/G_1\to G_{m-1}$ and it follows that $I(\mu)^F$ is injective as a $G_m/G_1$-module.  Hence by the above general remark $I((p^m-1)\delta+a\ep_n)=I(\lambda)\otimes I(\mu)^F$ is injective as a $G_m$-module.

(ii) Since the two modules in question  have the same $G$-socle $L((p^m-1)\delta+a\ep_n+\omega)$  it is enough to note that $I((p^m-1)\delta+a \epsilon_n)\otimes D$ is polynomially injective. 
This is clear from part (i) and Lemma 3.2(ii).

\end{proof}

\q We now show that the representation dimension of $S(n,r)$ grows with $r$.  We choose $m\geq 1$ such that $P=p^m>n$.  Let $h$ be a positive integer.  Suppose  that $r$ is a positive integer which is large enough so that 
$$r\geq ((P-1)|\delta|+1) \frac{P^h-1}{P-1}.$$
We write 
$$r- ((P-1)|\delta|+1) \frac{P^h-1}{P-1} =\sum_{i=0}^{h-1} P^i u_i +P^h u_h$$
with  $0\leq u_i\leq P-1$ for $0\leq i\leq h-1$ and $u_h\geq 0$.  Then
$$r=\sum_{i=0}^{h-1} P^i((P-1)|\delta|+1+u_i)+P^hu_h.$$

For $0\leq i\leq  h-1$ we define $\lambda^i\in X_m(n)$ by 
$$\lambda^i=\begin{cases} (1+u_i)\ep_1, & {\rm if} \q u_i<P-1;\cr
(P-n)\ep_1+\omega, & {\rm if}\q  u_i=P-1
\end{cases}$$
and put $\gamma=u_h\ep_1$. 

Then 
$$\mu=\sum_{i=0}^{h-1} P^i ((P-1)\delta+w_0\lambda^i)+ P^h\gamma$$
belongs to $\Lambda^+(n,r)$.   

\q By  Proposition 3.8,  Proposition 4.2(ii) and Lemma 4.1(ii)  we have that each $I((P-1)\delta+w_0\lambda^i)$ is admissible of positive index.  Note that $I(\gamma)=S^{u_h}E=\nabla(u_h\ep_1)$. (For a general quasi-hereditary algebra $A$ over $k$ with poset $\Lambda$ and $\lambda$ a maximal element of $\Lambda$ the corresponding costandard module $\nabla(\lambda)$ is injective, see [\cite{Do2}, Definition A 2.1].) Thus $I(\gamma)$ has   trivial endomorphism  algebra so is admissible of index $0$.  By Lemma 3.4(iii), to show that $I(\mu)$ is admissible of index at least $h$, it  is enough to prove that 
$$I(\mu)=\bigotimes_{i=0}^{h-1} I((P-1)\delta+w_0\lambda^i)^{F^{mi}}\otimes I(\gamma)^{F^{mh}}.$$
We prove first that 
$$I(\sum_{i=0}^{h-1} P^i ((P-1)\delta+w_0\lambda^i))=\bigotimes_{i=0}^{h-1} I((P-1)\delta+w_0\lambda^i)^{F^{mi}}\eqno{(\dagger)}.$$
By Proposition 4.2 (ii), for each $\lambda^i$ with $\lambda^i=(P-n)\epsilon_1+\omega$ we have that $I((P-1)\delta+w_o\lambda^i)=I((P-1)\delta+(P-n)\ep_n)\otimes D$.
Hence, extracting all such determinant factors, for a suitable non-negative integer $l$ we have
$$\bigotimes_{i=0}^{h-1} I((P-1)\delta+w_0\lambda^i)^{F^{mi}}=\bigotimes_{i=0}^{h-1} I((P-1)\delta+w_0\tau^i)^{F^{mi}}\otimes D^{\otimes l}$$
where $\tau^i=\lambda^i$ if $\lambda^i=(1+u_i)\epsilon_1$ and  $\tau^i=(P-n)\epsilon_1$ if  $\lambda^i=(P-n)\ep_1+\omega$. For each $\tau^i$ we have that $b_\padic(\tau^i)<p$, hence by Corollary 3.9 we get 
$$\bigotimes_{i=0}^{h-1} I((P-1)\delta+w_0\tau^i)^{F^{mi}}=I((P^h-1)\delta+w_0\sum_{i=0}^{h-1} P^i\tau^i).$$
Moreover by Proposition 4.2 (i), the last module is injective as $G_{hm}$-module. Therefore  $I((P^h-1)\delta+w_0\sum_{i=0}^{h-1} P^i\tau^i)\otimes D^{\otimes l}$ is injective as $G_{hm}$-module and since it has $G_{hm}$-socle $L((P^h-1)\delta+w_0\sum_{i=0}^{h-1} P^i \lambda^i)$ and  $(P^h-1)\delta+w_0\sum_{i=0}^{h-1} P^i \lambda^i\in X_{hm}(n)$ we get immediately by Lemma 3.2 (ii) that 
$$I((P^h-1)\delta+w_0\sum_{i=0}^{h-1} P^i\tau^i)\otimes D^{\otimes l}=I((P^h-1)\delta+w_0\sum_{i=0}^{h-1} P^i\lambda^i).$$
This completes the proof of $(\dagger)$.  Finally by Lemma 3.2 (iii) we get that 
$$I((P^h-1)\delta+w_0\sum_{i=0}^{h-1} P^i\lambda^i)\otimes  I(\gamma)^{F^{mh}}=I(\mu).$$

\bs

\q We summarise our findings.

\bs

\begin{theorem}

Choose $m\geq 1$ such that $P=p^m>n$. Let $h$ be a positive integer. Suppose that $r$ is a positive integer which is large enough so that
$$r\geq  ((P-1)|\delta|+1) \frac{P^h-1}{P-1}.$$
 Then there exists $\mu\in \Lambda^+(n,r)$ such that 
$\End_G(I(\mu))$ is isomorphic to a truncated polynomial algebra $k[x_1,\ldots,x_s]/(x_1^n,\ldots,x_s^n)$ with $s\geq h$.

\end{theorem}
\bs

Now if $I$ is a finite dimensional injective $S(n,r)$-module then the contravariant dual $Q$ of $I$ is projective and $\End_{S(n,r)}(Q)$ is isomorphic to the opposite algebra of $\End_{S(n,r)}(I)$.   However, we have, by \cite{MO}, Corollary 3.3, (which follows from  work of Bergh, Oppermann and Rouquier) 
 that if $S$ is a finite dimensional algebra over a field $k$ with a  projective module $Q$ whose endomorphism algebra has the form $k[x_1,\ldots,x_s]/(x_1^{n_1},\ldots,x_s^{n_s})$ with $n_1,\ldots,n_s>1$ then $S$ has representation dimension at least $s+1$. Hence we have the following.

 \bs

 \begin{corollary} Choose $m\geq 1$ such that $P=p^m> n$. Let $h$ be a positive integer.  If $r$ is a positive integer large enough so that 
 $$r\geq  ((P-1)|\delta|+1) \frac{P^h-1}{P-1}$$
  then $S(n,r)$ has representation dimension at least $h+1$.

 \end{corollary}

\section{Some remarks on the quantum case}

\subsection*{Generalities}

\q Now  let $k$ be a field and let $q$ be a non-zero element of $k$.  We consider the corresponding quantum general linear group $G_q(n)$, as in \cite{Do2}.   Further details of the framework and proofs or precise references for the results described below may be found in \cite{Do2}.  We have the bialgebra $A_q(n)$. As a $k$-algebra this is defined by generators $c_{ij}$, $1\leq i,j \leq n$, subject to certain quadratic relations (see e.g., \cite{Do2}, 0.22).  Comultiplication $\de:A_q(n)\to A_q(n)\otimes A_q(n)$ and the augmentation map $\ep:A_q(n)\to k$ are given by $\de(c_{ij})=\sum_{r=1}^n c_{ir}\otimes c_{rj}$ and $\ep(c_{ij})=\de_{ij}$, for $1\leq i,j\leq n$.  The algebra  $A_q(n)$ has a natural grading $A_q(n)=\bigoplus_{r=0}^\infty A_q(n,r)$ such that each  $c_{ij}$ has degree $1$. 
Each component $A_q(n,r)$ is a finite dimensional subcoalgebra  and the dual algebra is  the Schur algebra $S_q(n,r)$. 

\q The quantum determinant $d_q=\sum_{\pi\in \Sym(n)} \sgn(\pi) c_{1,1\pi}\ldots c_{n,n\pi}$ is a group-like element and $A_q(n)$ has an Ore localisation $A_q(n)_{d_q}$.  The bialgebra structure of $A_q(n)$ extends uniquely to a bialgebra structure on the localisation and  this localised bialgebra is in fact a Hopf algebra. The quantum general linear group $G_q(n)$ is the quantum group whose coordinate algebra  $k[G_q(n)]=A_q(n)_{d_q}$. 

\q We write $B_q(n)$ for the (quantum) subgroup of $G_q(n)$ whose defining ideal is generated by all $c_{ij}$ with $i<j$.   We write $T_q(n)$ for the subgroup of $G_q(n)$ whose defining ideal is generated by all $c_{ij}$ with $i\neq j$.  By a left (resp. right) module for a quantum group $G$ with coordinate algebra $k[G]$ we mean a right (resp. left) $k[G]$-comodule. By a $G$-module we mean a left $G$-module.
For each $\lambda=(\lambda_1,\ldots,\lambda_n)\in X(n)$ there is a one dimensional $B_q(n)$-module $k_\lambda$ with structure map $\tau: k_\lambda\to k_\lambda\otimes k[B_q(n)]$ taking $v\in k_\lambda$ to $v\otimes (c_{11}^{\lambda_1}c_{22}^{\lambda_2}\ldots c_{nn}^{\lambda_n}+I)$, where $I$ is the defining ideal of $B_q(n)$. Moreover the modules $k_\lambda$, $\lambda\in X(n)$, form a complete set of pairwise non-isomorphic simple $B_q(n)$-modules  and the restrictions of these modules to $T_q(n)$ form a complete set of pairwise non-isomorphic simple $T_q(n)$-modules. All $T_q(n)$-modules are completely reducible. An element $\lambda\in X(n)$ is a weight of a $T_q(n)$-module $V$ if it has a submodule isomorphic to $k_\lambda$.

\q Given a subgroup $H$ of a quantum group $G$ over $k$ and an $H$-module $V$ we have the induced $G$-module $\ind_H^GV$. For $\lambda\in X(n)$ the induced module $\ind_{B_q(n)}^{G_q(n)}k_\lambda$ is non-zero if and only if $\lambda\in X^+(n)$. We set $\nabla_q(\lambda)=\ind_{B_q(n)}^{G_q(n)}k_\lambda$, for $\lambda\in X^+(n)$.  The module $\nabla_q(\lambda)$ has a unique irreducible submodule which we denote $L_q(\lambda)$. The modules $L_q(\lambda)$, $\lambda\in X^+(n)$, form a complete set of pairwise non-isomorphic irreducible $G_q(n)$-modules.

\q Suppose $G$ is a quantum group over $k$ and $V$ is a $G$-module with structure map $\tau:V\to V\otimes k[G]$ and basis $v_i$, $i\in I$. The  corresponding coefficient elements $f_{ij}\in k[G]$, are defined by the equations
$$\tau(v_i)=\sum_{j\in I} v_j\otimes f_{ji}.$$
The coefficient space $\cf(V)$ of $V$ is the $k$-span of all $f_{ij}$ (it is independent of the choice of basis).  A $G_q(n)$-module $V$ is polynomial (resp. polynomial of degree $r$) if $\cf(V)\leq A_q(n)$ (resp. $\cf(V)\leq A_q(n,r)$).  A $G_q(n)$-module which is polynomial of degree $r$ may be regarded as an $A_q(n,r)$-comodule and hence as a module for the dual algebra $S_q(n,r)$.  In this way one has equivalences of categories between the category of polynomial $G_q(n)$-modules  of degree $r$, the category of right $A_q(n,r)$-comodules  and the category of left $S_q(n,r)$-modules.

\q Taking $q=1$ one recovers the classical case of the general linear group scheme and its representation theory. We shall write $G(n)$ for $G_q(n)$, write $B(n)$ for $B_q(n)$ and write $T(n)$ for $T_q(n)$ in this case. Further,  the coordinate function $c_{ij}$ will be denoted $x_{ij}$ in this case. 


\q If $q$ is not a root of unity or $k$ has characteristic $0$ and $q=1$,  then all $G_q(n)$-modules are completely reducible. We shall assume from now on that $q$ is a primitive $l$th root of unity, with $l>1$.  There is a Hopf algebra homomorphism $F^\sharp:k[G(n)]\to k[G_q(n)]$ taking $x_{ij}$ to $c_{ij}^l$, for $1\leq i,j\leq n$. The Frobenius morphism $F:G_q(n)\to G(n)$ is the quantum group morphism whose comorphism is $F^\sharp$.  By abuse of notation we also write $F:B_q(n)\to B(n)$ for the restriction of the Frobenius morphism.
The infinitesimal group scheme $G_{\inf,q}(n)$ is the subgroup scheme of $G_q(n)$ whose defining ideal is generated by the elements $c_{ij}^l-\de_{ij}$, $1\leq i,j\leq n$.   For a $G(n)$-module $V$, with structure map $\tau: V\to V\otimes k[G(n)]$ we write $V^F$ for the  vector space $V$ now regarded as a $G_q(n)$-module via the structure map $(\id_V \otimes F^\sharp)\circ\tau :V\to V\otimes k[G_q(n)]$, where $\id_V$ denotes the identity map on $V$.  
We write $X_{\inf,q}(n)$ for the set of column $l$-regular weights, i.e., the set of $\lambda=(\lambda_1,\ldots,\lambda_n)$ such that $0\leq \lambda_1-\lambda_2,\lambda_2-\lambda_3,\ldots,\lambda_{n-1}-\lambda_n,\lambda_n<l$. The modules $L_q(\lambda)$, $\lambda\in X_{\inf,q}(n)$, form a complete set of pairwise non-isomorphic irreducible $G_{\inf,q}(n)$-modules, see \cite{Do2}, Section 3.2.  Moreover, these modules are Schurian, in the sense that $\End_{G_q(n)}(L_q(\lambda))=k$, for $\lambda\in X^+(n)$.

\q To save on notation we shall try to suppress $n$ and $q$ where  possible. We shall write $G$ for $G_q(n)$, write $B$ for $B_q(n)$,  write $G_\inf$ for $G_{\inf,q}(n)$ and write  $X_\inf(n)$ for $X_{\inf,q}(n)$. We shall write $BG_\inf$ for the  thickening of $B$, i.e.,   quantum subgroup of $G$ whose defining ideal is generated by the elements $c_{ij}^l$, $1\leq i<j  \leq n$.  We shall write $\nabla(\lambda)$ for $\nabla_q(\lambda)$ and $L(\lambda)$ for $L_q(\lambda)$, $\lambda\in X^+(n)$. 

\q We shall write $\barG$ for $G(n)$ and $\barB$ for $B(n)$. Moreover, for $\lambda\in X^+(n)$ we shall write $\barnabla(\lambda)$ for $\ind_\barB^\barG k_\lambda$ and $\barL(\lambda)$ for the socle of $\barnabla(\lambda)$. An  element $\lambda$ of $X^+(n)$ may be written uniquely in the form $\lambda=\lambda^0+l\barlambda$, with $\lambda^0\in X_\inf(n)$ and $\barlambda\in X^+(n)$.  Moreover one has $L(\lambda)\cong L(\lambda^0)\otimes \barL(\barlambda)^F$ (Steinberg's tensor product theorem, see \cite{Do2}, Section 3.2).

\q For $\lambda\in X(n)$ we shall write $Z'(\lambda)$ for the induced module $\ind_B^{BG_\inf}  k_\lambda$. The  modules $Z'(\lambda)$ have properties analogous  to the corresponding modules for reductive algebraic groups,  see \cite{Jan}, II. Chapter 3. These modules are considered by Cox in   \cite{Cox}, Section 3.  There is no explicit statement in   \cite{Cox}, Section 3  that these modules are finite dimensional. In a fuller treatment one would include the precise result, namely that the dimension is $l^{{n\choose 2}}$ but we content ourselves here with finite dimensionality.  So we show that $\ind_B^{BG_\inf}M$ is finite dimensional for $M$ a finite dimensional $B$-module.  It suffices to show that the  $G_\inf$-socle  is finite dimensional since $G_\inf$ is a finite quantum group. Hence it suffices to show that 
$\Hom_{G_\inf}(L,\ind_B^{BG_\inf}M)$ is finite dimensional, for $L=L_q(\lambda)$, $\lambda\in X_{\inf,q}(n)$.  Now we have 
\begin{align*}
\Hom_{G_\inf}(L,\ind_B^{BG_\inf}M&)=H^0(G_\inf,L^*\otimes \ind_B^{BG_\inf}M)\cr
&=H^0(G_\inf,\ind_B^{BG_\inf}(L^*\otimes M))
\end{align*}
using the tensor identity.  Hence it suffices to prove that \\
$H^0(G_\inf,\ind_B^{BG_\inf}M)$ is finite dimensional for $M$ finite dimensional.  By left exactness if suffices to prove that  $H^0(G_\inf,\ind_B^{BG_\inf}k_\lambda)=H^0(G_\inf,Z'(\lambda))$ is finite dimensional, for $\lambda\in X(n)$.  Also, for $\lambda=\lambda^0+l\barlambda$, with $\lambda^0\in X_\inf(n)$, $\barlambda\in X(n)$ we have $Z'(\lambda)=Z'(\lambda^0)\otimes k_{l\barlambda}$, by \cite{Cox},  Theorem 3.4 (i). Hence we may assume $\lambda\in X_\inf(n)$.

\q Now by \cite{DoStd},  Proposition 1.5, (i) (or the original source \cite{PW},  Section 2.10), $H^0(G_{\inf},Z'(\lambda))=V^F$, for some $\barB$-module $V$.   We may assume $H^0(G_\inf,Z'(\lambda))$ to be non-zero. Let $U$ be a non-zero   finite dimensional submodule of $V$. Then 
 by Frobenius reciprocity we have
$$\Hom_{BG_\inf}(U^F,Z'(\lambda))=\Hom_B(U^F,k_\lambda),$$
  This is non-zero, (since the left hand side contains the inclusion of $U^F$ in $Z'(\lambda)$). Thus we have 
   $\lambda=0$ and so $\Hom_{BG_\inf}(U^F,Z'(\lambda))=\Hom_B(U^F,k)$.
  Note that Frobenius reciprocity gives an embedding of the  trivial  module in $k$. Let $Z_0$ be a copy of $k$ in $Z'(0)$.  Then we have $\dim \Hom_{BG_\inf}(U^F,Z_0)=\dim \Hom_{BG_\inf}(U^F,Z'(0))$. Hence every homomorphism from $U^F$ into $Z'(0)$ goes into $Z_0$, in particular the inclusion $U^F\to Z'(0)$ has image in $Z_0$. Hence $U^F=Z_0$, and so the $G_\inf$ socle of $Z'(0)$ is $k$ and we are done.

\subsection*{The characteristic $0$ case}

\q We now  consider the case in which $k$ has characteristic $0$.

\bs

\begin{lemma} If $k$ has characteristic $0$ then the  set of decomposition numbers $[\nabla(\lambda):L(\mu)]$, $\lambda,\mu\in X^+(n)$,  is bounded above.
\end{lemma}

\begin{proof} Let $N$ be greater or equal to the dimension of $\dim Z'(\lambda)$, for all  $\lambda\in X_\inf(n)$. For $\lambda=\lambda^0+l\barlambda\in X(n)$ with $\lambda^0\in X_\inf(n)$, $\barlambda\in X(n)$ we have $Z'(\lambda)=Z'(\lambda)\otimes k_{l\barlambda}$, by \cite{Cox}, Theorem 3.4, (i) so that $N$ is an upper bound for all $\dim Z'(\lambda)$, $\lambda\in X(n)$. 

\q Let $\lambda\in X^+(n)$.  We have 
$$\nabla(\lambda)=\ind_B^G k_\lambda=\ind_{BG_\inf}^G\ind_B^{BG_\inf} k_\lambda=\ind_{BG_\inf}^G Z'(\lambda)$$
 by the transitivity of induction.   Let $0=X_0<X_1<\cdots < X_t=Z'(\lambda)$ be a $BG_\inf$-composition series. Thus $t\leq N$.  By left exactness of induction we have
 $$[\nabla(\lambda):L(\mu)]\leq \sum_{i=1}^t [\ind_B^{BG_\inf} X_i/X_{i-1}: L(\mu)].$$
 A factor $X_i/X_{i-1}$ has the form $L_q(\xi)\otimes k_{l\nu}$ for some $\xi\in X_\inf(n)$, $\nu\in X(n)$.  Thus we have
 $$\ind_{BG_\inf}^G X_i/X_{i-1}=L_q(\xi)\otimes  \ind_{BG_\inf}^G k_{l\nu}=L_q(\xi)\otimes(\ind_\barB^\barG k_\nu)^F$$ 
 by \cite{Cox}, Lemma 4.6. Now $\ind_\barB^\barG k_\nu=\barnabla(\nu)$ for $\nu\in X^+(n)$ and is otherwise $0$. Moreover $\barnabla(\nu)=\barL(\nu)$, for $\nu\in X^+(n)$, since $\barG$-modules are completely reducible.  Thus by Steinberg's tensor product theorem, $\ind_{BG_\inf}^G X_i/X_{i-1}$ is either irreducible or $0$. Hence we have $[\nabla(\lambda):L(\mu)]\leq t\leq N$.
  This completes the proof of the Lemma.

\end{proof}

\begin{proposition} For fixed $n$, a field $k$ of characteristic $0$ and $0\neq q\in k$ the representation dimension of the Schur algebra $S_q(n,r)$ (as $r$ varies) is bounded above.

\end{proposition}

\begin{proof}  If $A_1$ and $A_2$ are finite dimensional algebras then the representation dimension of $A_1 \times A_2$ is the maximum of the representation dimensions of $A_1$ and $A_2$. Hence it suffices to prove that there exists a constant $K$ such that the representation dimension of every block of $S_q(n,r)$  is less than  $K$.  So let $A$ be a block ideal of $S_q(n,r)$.  Let $\Theta$ be the subset of $\Lambda^+(n,r)$ such that  the simple  $S_q(n,r)$ modules belonging to the block $A$ are the modules   $L(\lambda)$, $\lambda\in \Theta$.  Let $P(\lambda)$ be the projective cover of $L(\lambda)$. Then the block $A$ is Morita equivalent to its basic algebra $A'=\End_{S_q(n,r)}(\bigoplus_{\lambda\in \Theta} P(\lambda))^\opp$.  Representation dimension is a Morita invariant so that it suffices to prove that there exists a constant $K$ such that the representation dimension of $A'$ is less than $K$, for all blocks $A$.  But now, from Iyama's result, \cite{I}, 1.2 Corollary, the representation dimension of a finite dimensional $k$-algebra $R$ is finite and bounded by a function of the  dimension of $R$ as a $k$-vector space.  Thus it suffices to prove that there exists  a uniform bound on the dimension of $A'$ as $A$ ranges over all blocks of $S_q(n,r)$ (as $r$ varies).

\q Now as a $k$-vector space $A'=\End_{G_q(n)}(\bigoplus_{\lambda\in \Theta} P(\lambda))^\opp$ is isomorphic to $\bigoplus_{\lambda,\mu\in\Theta}\Hom_{G_q(n)}(P(\lambda),P(\mu))$. Hence we have $\dim A'\leq |\Theta|^2 M$, where $M$ is the maximal dimension of the spaces $\Hom_G(P(\lambda),P(\mu))$, as $\lambda,\mu$ vary over $\Theta$. 
Moreover by a result of Cox, \cite{Cox},  Theorem 5.3, we have $|\Theta|\leq |S_n|=n!$. Hence it suffices to prove that there is a uniform bound on $\dim \Hom_G(P(\lambda),P(\mu))$, as $\lambda,\mu$ vary over $\Lambda^+(n)$.  However,  for $r\geq 0$ the algebra $S_q(n,r)$ is a quasi-heredity algebra with standard modules $\Delta(\lambda)$, $\lambda\in \Lambda$ (and the dominance order on $\Lambda^+(n,r)$).  The costandard modules are the induced modules $\nabla(\lambda)$, $\lambda\in \Lambda^+(n,r)$,  and for each $\lambda\in \Lambda^+(n,r)$ the modules $\Delta(\lambda)$ and $\nabla(\lambda)$ have the same composition factors,  counting multiplicities, see e.g., \cite{DoStd}, Section 4.  Hence we have
\begin{align*}
\dim\, &\Hom_{S_q(n,r)}(P(\lambda),P(\mu))\cr
&=[P(\mu):L(\lambda)]\cr
&=\sum_{\tau\in\Lambda^+(n,r)} (P(\lambda):\Delta(\tau))[\Delta(\tau):L(\mu)]\cr
&=\sum_{\tau\in \Lambda^+(n,r)} [\nabla(\tau):L(\lambda)][\Delta(\tau):L(\mu)]\cr
&=\sum_{\tau\in \Lambda^+(n,r)} [\nabla(\tau):L(\lambda)][\nabla(\tau):L(\mu)]\cr
&\leq |\Theta| N^2\leq n! N^2
\end{align*}
where $N$ is an upper bound for all decomposition numbers $[\nabla(\xi):L(\nu)]$, for $\xi,\nu\in X^+(n)$ (and the existence of such an $N$ is guaranteed by the Lemma above). This completes the proof.

\end{proof}

\subsection*{The positive characteristic case}

\q Now suppose that $k$ has characteristic $p>0$.   We claim that the main development  of Sections 1-4 above goes through in this case.

\q  In addition to the generalities  on quantum general linear groups discussed in Section 5.1, we shall need the $q$-analogues of various results described above in the classical case. Our reference for the background results is \cite{Do2}.  Let $m\geq 1$. In this section we shall write $G_m'$ for the (quantum) subgroup scheme of $G$ with defining ideal generated by $c_{ij}^{lp^m}-\de_{ij}$, $1\leq i,j\leq n$. In this section $X_m(n)$ is the set of all $\lambda\in X^+(n)$ expressible in the form $\lambda=\lambda^0+p\barlambda$ with $\lambda^0$ an $l$ column regular weight and $\barlambda$ a $p^m$ column regular weight. 
The modules $L(\lambda)$, $\lambda\in X_m(n)$,  form a complete set of pairwise non-isomorphic $G_m'$-modules (this follows from \cite{Cox}, Lemma 3.1 and Steinberg's tensor product theorem).  The arguments of Section 3 easily adapt to the present context, in particular one may prove a suitable version of Lemma 3.2  (we leave the details to the interested reader).

\q We pick $m\geq 1$ such that $P=p^m>n$. Let $r$ be a positive integer.  If $r\geq (lP-1)|\delta|+1$ we  write $r=(lP-1)|\delta|+1+u_{-1}+lPs$, with $0\leq u_{-1}\leq lP-1$.  Let $h$ be a positive integer. Suppose that $s$ is a positive integer which is large enough so that 
 $s > ((P-1)|\delta|+1)\frac{P^h-1}{P-1}$.  We write
$$s- ((P-1)|\delta|+1)\frac{P^h-1}{P-1}=\sum_{i=0}^{h-1}P^iu_i+P^hu_h$$
with $0\leq u_i\leq P-1$, for $0\leq i\leq h-1$ and $u_h\geq 0$.  We define 
$$\lambda^i=
\begin{cases} (1+u_i)\ep_i,  &{\rm if}\q  u_i<P-1;\cr
(P-n)\ep_1+\omega, &{\rm if}\q  u_i= P-1
\end{cases}
$$
and $\gamma=u_h\ep_1$. 

\q Then $\mu=((lP-1)\delta+w_0\lambda_{-1})+lP\sum_{i=0}^{h-1}((P-1)|\delta|+w_0\lambda^i)+lP^{h+1}\gamma$ belongs to $\Lambda^+(n,r)$. By the arguments above (using suitable references to \cite{Do2}), one has that $\End_{G_q(n)}(I(\mu))$ is a tensor product of at least $h+1$ copies of $k[x]/(x^n)$.  Hence we obtain the following result.

\bigskip

\begin{theorem} Suppose $k$ is a field of characteristic $p>0$ and $q$ is a primitive $l$th root of unity, with $l>1$.
Choose $m\geq 1$ such that $P=p^m> n$. Let $h$ be a positive integer. If $r$ is a positive integer large enough so that 
 $$r\geq  ((lP-1)|\delta|+1)+ lP((P-1)|\delta|+1)\frac{P^h-1}{P-1}$$
 then $S_q(n,r)$ has representation dimension at least $h+2$.

\end{theorem}

\end{document}